\documentclass{amsart} 
\usepackage{graphicx} 
\usepackage{newlattice}

\newtheorem{theorem}{Theorem}
\newtheorem{lemma}[theorem]{Lemma}

\theoremstyle{definition}

\newcommand{\al}{a_\tup{l}}
\newcommand{\ar}{a_\tup{r}}

\newcommand{\br}{b_\tup{r}}

\newcommand{\qdn}{\fq_{\tup{dn}}}
\newcommand{\qup}{\fq_{\tup{up}}}
\newcommand{\jcon}{\textup{J}(\Con L)}

\newcommand{\swing}{\mathbin{\raisebox{2.0pt}
       {\rotatebox{160}{$\curvearrowleft$}}}}

\DeclareMathOperator{\length}{length}

\newcommand{\grrel}{\mathbin{\gr}}

\begin{document}
\title[The Swing Lemma]{Congruences in slim, planar, \\semimodular lattices: The Swing Lemma} 
\author{G. Gr\"{a}tzer} 
\address{Department of Mathematics\\
   University of Manitoba\\
   Winnipeg, MB R3T 2N2\\
   Canada}
\email[G. Gr\"atzer]{gratzer@me.com}
\urladdr[G. Gr\"atzer]{http://server.math.umanitoba.ca/homepages/gratzer/}

\date{April 21, 2015}
\keywords{Prime-perspective, congruence, congruence-perspective, perspective, prime interval.}
\subjclass[2010]{Primary: 06C10, 06B10}

\begin{abstract}
In an earlier paper, 
to describe how a congruence spreads 
from a prime interval to another in a finite lattice,
I introduced the concept 
of prime-perspectivity and its transitive extension,
prime-projectivity and proved the Prime-projectivity Lemma.

In this paper, I specialize the Prime-projectivity Lemma
to slim, planar, semimodular lattices to obtain the Swing Lemma,
a very powerful description 
of the congruence generated by a prime interval
in this special class of lattices.
\end{abstract}

\maketitle

\section{Introduction}\label{S:Introduction}
To describe how a congruence spreads 
from a prime interval to another in a finite lattice $L$,
I introduced the concept of prime-perspectivity in \cite{gG14c}.

Let $L$ be a finite lattice 
and let $I$ and $J$ be intervals of $L$. 
Figure~\ref{F:intro} depicts 
the binary relation \emph{$I$ down-perspective to $J$}, 
in formula, $I \perspdn J$. 
We define dually the binary relation \emph{$I$ up-perspective to $J$}, 
in formula, $I \perspup J$. 
Finally, let $I$ be \emph{perspective to $J$}, in formula,
$I \persp J$, if $I \perspdn J$ or $I \perspup J$.

\begin{figure}[h]
\centerline{\includegraphics{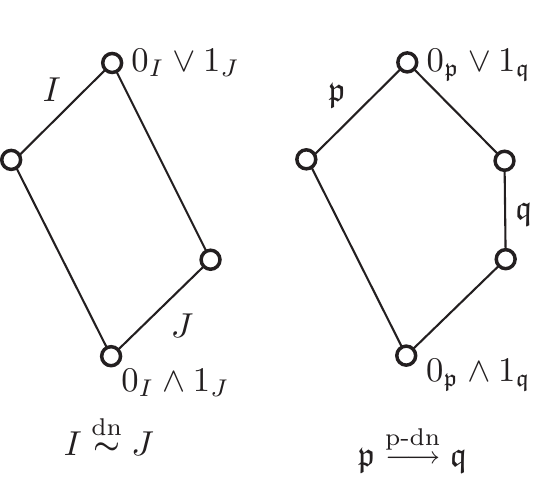}}
\caption{Introducing prime-perspectivity}\label{F:intro}
\end{figure}

Now let $\fp$ and $\fq$ be prime intervals of $L$.
In the second diagram in Figure~\ref{F:intro},
$\fq$~is collapsed by $\con{\fp}$,
but we cannot get from $\fp$ to $\fq$ 
by a sequence of down- and up-perspectivities between prime intervals. 
So~we introduce a more general step between two prime intervals: 
$\fp$~is \emph{prime-perspective down} to $\fq$ 
(in formula, $\fp \pperspdn \fq$) if
$\fp$ is down-perspective to~$[0_\fp \mm 1_\fq, 1_\fq]$
and $\fq$ is contained in $[0_\fp \mm 1_\fq, 1_\fq]$. 
If $\fp \pperspdn \fq$, then $\fp$ and $\fq$ generate
an $\SN 5$, as in the second diagram of Figure~\ref{F:intro},
or a $\SB 2 = \SC 2^2$, as in the first diagram of Figure~\ref{F:intro},
or $\SC 2$, if $\fp = \fq$.

We define \emph{prime-perspective up}, in formula,
$\fp \pperspup \fq$, dually.
Let \emph{prime-perspective}, in formula, $\fp \ppersp \fq$, 
mean that $\fp \pperspup \fq$ or $\fp \pperspdn \fq$ and let
\emph{prime-projective}, in~formula, $\fp \pproj \fq$, 
be the transitive extension of $\ppersp$. 

Now we state the main result of G. Gr\"atzer~\cite{gG14c}: 
we only have to go through prime intervals by prime-perspectivities
to spread a congruence from a prime interval to another
in a finite lattice.
 
\begin{named}{Prime-projectivity Lemma}
Let $L$ be a finite lattice and 
let $\fp$ and $\fq$ be distinct prime intervals in $L$.  
Then $\fq$ is collapsed by $\con{\fp}$ 
if{}f $\fp \pproj \fq$, that is, 
if{}f there exists a sequence of pairwise distinct prime intervals
$\fp = \fu_0, \fu_1, \dots, \fu_n = \fq$ satisfying
\begin{equation}\label{E:ppthm}
\fp = \fu_0 \ppersp \fu_1 \ppersp \dotsm \ppersp \fu_n = \fq.
\end{equation}
\end{named}

Let us call a lattice $L$ an \emph{SPS lattice}, 
if it is slim (contains no $\SM 3$ sublattice), 
planar, and semimodular. Note that an SPS lattice is finite by definition.

For the prime intervals $\fp, \fq$ of an SPS lattice $L$, 
we define a new binary relation:
$\fp$~\emph{swings} to $\fq$, in formula, $\fp \swing \fq$,
if $1_\fp = 1_\fq$, 
the element $1_\fp = 1_\fq$ covers at least three elements,
and $0_\fq$ is neither the left-most nor the right-most element
covered by $1_\fp = 1_\fq$. 
We say that $\fp \swing \fq$ is \emph{established by} 
an $\SN 7$ sublattice of $L$, 
if the $\SN 7$ is generated by $0_\fp$, $0_\fq$, 
and a third element covered by $1_\fp$.

See Figure~\ref{F:n7+} for two examples. 

\begin{figure}[b]
\centerline{\includegraphics{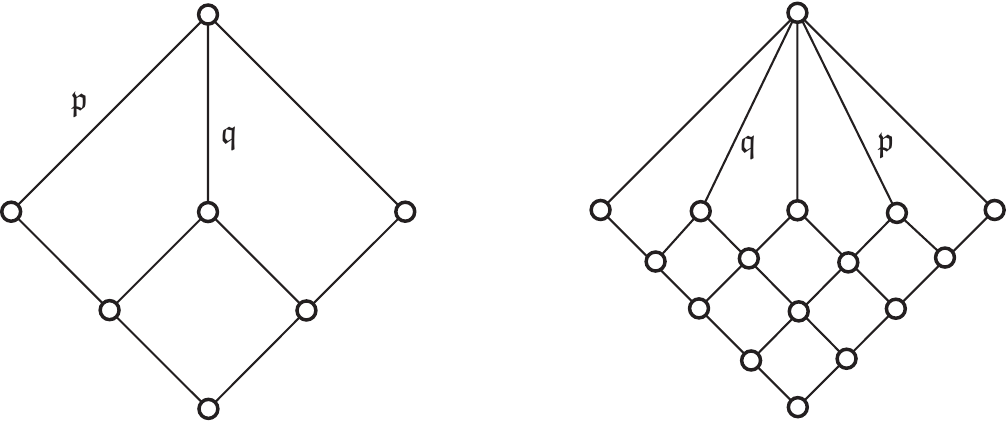}} 
\caption{Swings, $\fp \protect\swing \fq$}\label{F:n7+}
\end{figure}

\begin{named}{Swing Lemma}
Let $L$ be an SPS lattice 
and let $\fp$ and $\fq$ be distinct prime intervals in $L$. 
Then  $\fq$ is collapsed by $\con{\fp}$ if{}f 
there exists a prime interval $\fr$ 
such that $\fp$ is up-perspective to $\fr$
and there exists a sequence of prime intervals
and a~sequence of binary relations 
\begin{equation}\label{Eq:sequence}
\fr = \fr_0 \grrel_1 \fr_1 \grrel_2 \fr_2 
      \dots \grrel_n \fr_n = \fq,
\end{equation}
where each relation $\grrel_i$ is $\perspdn$ or $\swing$.

In~addition, the sequence \eqref{Eq:sequence} also satisfies 
\begin{equation}\label{E:geq}
   1_{\fr_0} \geq 1_{\fr_1} \geq \dots \geq 1_{\fr_n}.
\end{equation}
\end{named}

If we choose a shortest sequence in the Swing Lemma, 
then the prime intervals 
$\fr = \fr_0, \fr_1, \fr_2, \dots , \fr_n = \fq$
are pairwise distinct and the down-perspectivities alternate with the swings.

The Swing Lemma is easy to visualize. 
Up-perspectivity is ``climbing up'', 
down-perspectivity is ``sliding down''. 
So we get from $\fp$ to $\fq$ by climbing up once,
and then alternating sliding down and swinging.

In this paper we give an elementary proof of this result.
``Elementary'' means that 
we do not use the deep techniques and results
developed for rectangular lattices in G.~Cz\'edli~\cite{gC13}. 
An alternative proof of the Swing Lemma 
can be found in G.~Cz\'edli~\cite{gC14b}.

\section{Preliminaries}\label{S:Prelim}

\subsection{Fork construction}\label{S:forks}
The following lemma is implicitly used 
in G.~Cz\'edli and E.\,T.~Schmidt~\cite{CS12a}:

\begin{lemma}\label{L:fork}
Let $K$ be an SPS lattice. 
Let $S =\set{o, \al, \ar, t}$ be a covering square of~$K$,
and let $\al$ be to the left of $\ar$. 
Then there are maximal chains
\begin{align*}
   \al &= x_{l,1} \succ x_{l,2} \succ \dots \succ x_{l,n_l},\\
    o &= y_{l,1} \succ y_{l,2} \succ \dots \succ y_{l,n_l},
\end{align*}
such that $x_{l,n_l}$ and $y_{l,n_l}$ 
are on the left boundary of $K$
and the interval $[y_{l,n_l}, \al]$ is 
isomorphic to $\SC 2 \times \SC{n_l}$,
and symmetrically.

Let 
\begin{align*}
G[S] = S &\uu \set{x_{l,1}, x_{l,2},\dots, x_{l,n_l}}
 \uu \set{y_{l,1}, y_{l,2}, \dots, y_{l,n_l}}\\
 &\uu \set{x_{r,1}, x_{r,2}, \dots, x_{r,n_r}}
 \uu \set{y_{r,1}, y_{r,2}, \dots, y_{r,n_r}}.
\end{align*}
Then $G[S]$ is a join-subsemilattice of $K$. 
Furthermore, $K$ is a cover-preserving extension of~$G[S]$.
\end{lemma}

\begin{figure}[h]
\centerline{\includegraphics{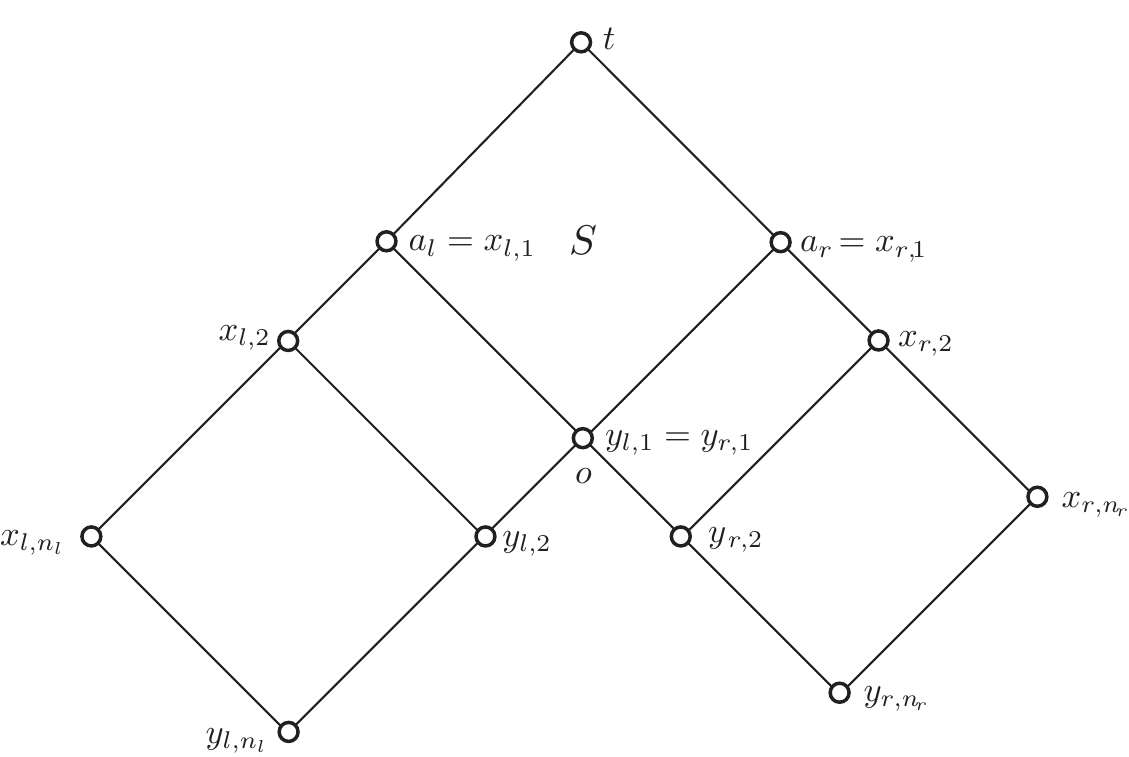}}
\caption{$G[S]$, a join-subsemilattice of $K$}\label{F:forklemma}
\end{figure}

As in G.~Cz\'edli and E.\,T.~Schmidt~\cite{CS12a}, 
\emph{inserting a fork} into $K$ 
at the covering square~$S$ adds the elements
\begin{equation}\label{E:FS}
F[S] = \set{m,z_{l,1} \succ \dots \succ z_{l,n_l}, 
          z_{r,1} \succ \dots \succ z_{r,n_r}},
\end{equation}
see Figure~\ref{F:forkdetails}, so that the interval
$[o,t]$ turns into an $\SN 7$ (see Figure~\ref{F:s7}),
and the interval $[y_{l,n_l}, \al]$ becomes 
isomorphic to $\SC 3 \times \SC{n_l}$,
and symmetrically. Let $K[S]$ denote this construct.
Then $K[S]$ is an SPS lattice, 
as~observed in G.~Cz\'edli and E.\,T.~Schmidt~\cite{CS12a}.
See Figure~\ref{F:forkexample} for an illustration; 
the black filled elements form~$F[S]$.

\begin{figure}[t]
\centerline{\includegraphics{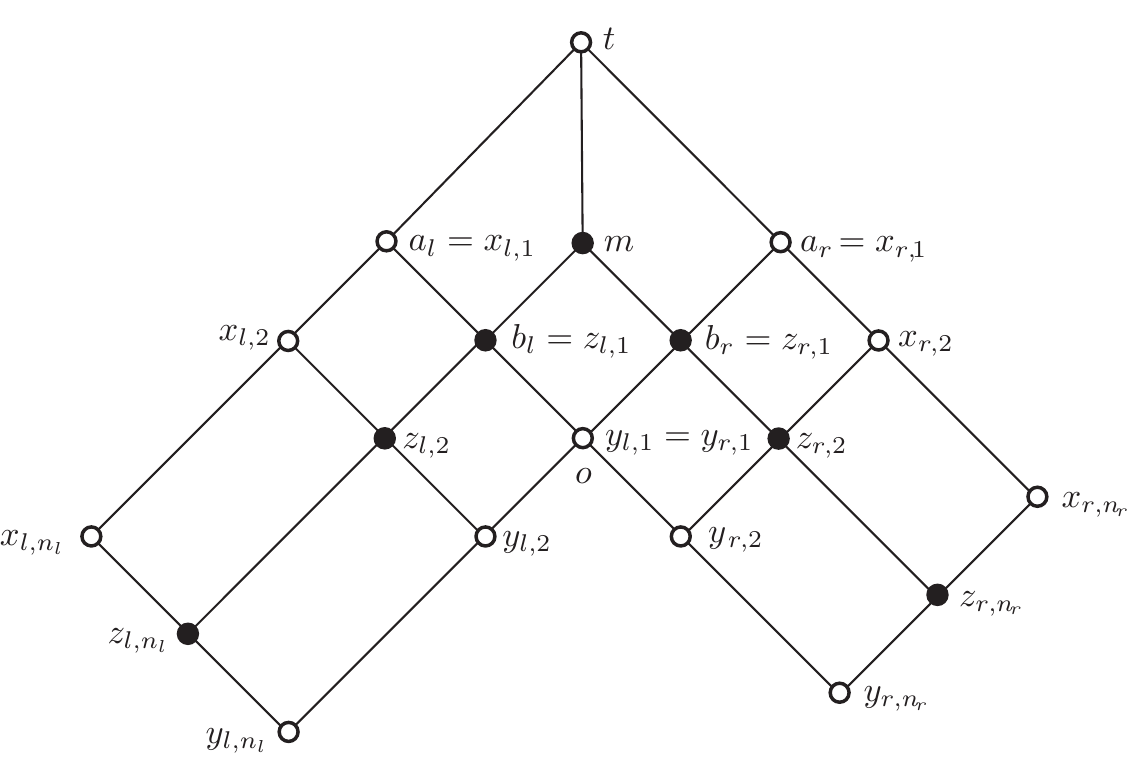}}
\caption{$F[S]$ inserted}\label{F:forkdetails}

\bigskip

\bigskip

\centerline{\includegraphics{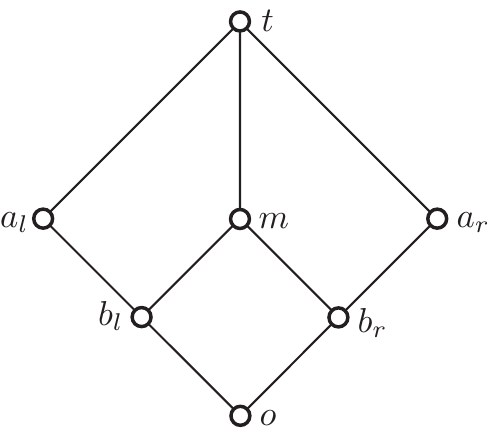}}
\caption{The lattice $\SfS 7$}\label{F:s7}
\end{figure}

\begin{figure}[t]\centerline{\includegraphics{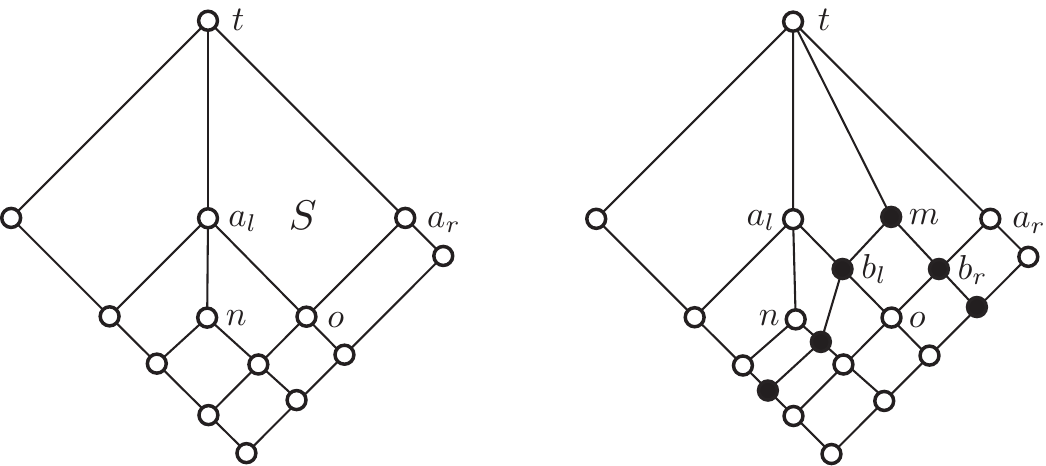}}
\caption{Inserting a fork: $K$, $S$, and $K[S]$}\label{F:forkexample}
\end{figure}

\subsection{SPS lattices}\label{S:SPS}

For an overview of this topic, 
see G. Cz\'edli and G.~Gr\"atzer \cite{CGa},
Chapter 3 of G. Gr\"atzer and F. Wehrung eds.~\cite{LTE}.

Let us call the elements $u,v,w \in L$ 
\emph{pairwise disjoint over the element $a$}
provided that $a = u \mm v = v \mm w = w \mm u$.

The first, third, and fourth statement of the next lemma 
can be found in the literature 
(see G.~Gr\"atzer and E.~Knapp~\cite{GKn07}--\cite{GKn10}, 
G.~Cz\'edli and E.\,T. Schmidt~\cite{CS12}--\cite{CS12a}).

\begin{lemma}\label{L:known}
Let $L$ be an SPS lattice. 
\begin{enumeratei}
\item An element of $L$ has at most two covers.
\item If the elements $u,v,w \in L$ 
are pairwise disjoint over $a$,
then two of them are comparable.
\item Let $x \in L$ 
cover three distinct elements $u$, $v$, and $w$.
Then the set $\set{u,v,w}$ generates an $\SfS 7$ sublattice.
\item If the elements $u$, $v$, and $w$ are adjacent, 
then the $\SfS 7$ sublattice of \tup{(iii)} 
is a cover-preserving sublattice.
\item  Let $\fp, \fq$ be distinct prime intervals of $L$. 
If $\fp \swing \fq$, then $0_\fq$ is a meet-irreducible element.
\end{enumeratei}
\end{lemma}

\begin{proof}
To verify (ii), let the elements $u,v,w \in L$ 
be pairwise disjoint over the element~$a \in L$.
By way of contradiction, 
assume that no two of them are comparable.
Then, in particular, the elements $u,v,w$ are pairwise distinct.
So~$a < u$, $a < v$, $a < w$. 
We can choose $a \prec u' \leq u$, 
$a \prec v' \leq v$, $a \prec w' \leq w$.
The elements $u',v',w'$ are pairwise distinct.
Indeed, if say, $u' = v'$, then $u' = v' \leq u \mm v$,
contradicting that $a = u \mm v$.
So the elements $u',v',w'$ are pairwise distinct
and cover~$a$, contradicting~(i). 
Finally, (v) follows from (iv).
\end{proof}

Lemma~\ref{L:known}(i) and (ii) state in different ways 
that there are only two directions ``to go up'' from an element.
The next lemma states this in one more way.
This important statement follows from \cite[Lemma 2.8]{CS11}.

\begin{lemma}\label{L:persp}
Let $L$ be an SPS lattice. 
Let $\fq, \fq_1,\fq_2$ be pairwise distinct prime intervals of $L$
satisfying $\fq_1 \perspdn \fq$ and $\fq_2 \perspdn \fq$.
Then $\fq_1 \persp \fq_2$.
\end{lemma}

Lemma~\ref{L:persp} can also be derived from Lemma~\ref{L:known}, 
see the arXiv version of this paper. 

An SPS lattice $L$ is called a \emph{slim patch lattice} 
if it has exactly two dual atoms that meet in $0$.
For a slim patch lattice~$L$, we shall use the notation:
$\fp_\tup{l}$ and $\fp_\tup{r}$ are the two prime intervals 
on the top boundaries of~$L$ and
$\fp_\tup{l} = [c_\tup{l}, 1_L]$ on the left, 
$\fp_\tup{r} = [c_\tup{r}, 1_L]$ on the right. 

The following result can be found in
G. Cz\'edli and E.\,T. Schmidt~\cite{CS12}.

\begin{named}{Structure Theorem for Slim Patch Lattices}
Let $L$ be a slim patch lattice. 
Then we can obtain $L$ 
from the $4$-element Boolean lattice $\SB 2 = \SC 2^2$
by a series of fork insertions.
\end{named}

\section{Two lemmas}\label{S:newlemma}
\begin{figure}[b!]
\centerline{\includegraphics{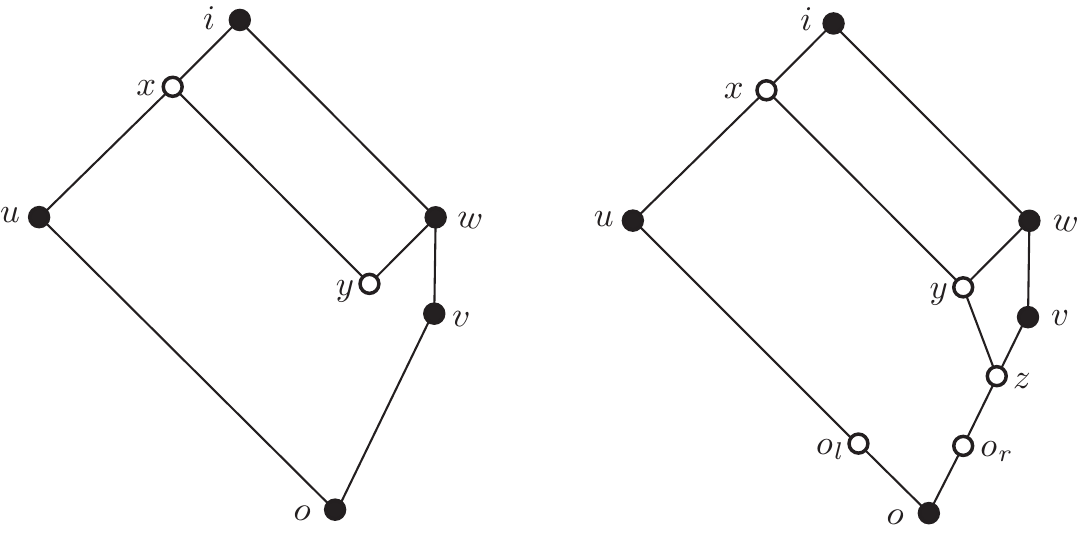}}
\caption{The elements for Lemma~\ref{L:newlemma}}
\label{F:newlemma}
\end{figure}

The following lemma is a crucial step 
in the proof of the Swing Lemma.

\begin{lemma}\label{L:newlemma}
Let $L$ be an SPS lattice. 
Let $N = \set{o, u, i, v, w}$ be an $\SN 5$ sublattice of~$L$,
with $o < u < i$ and $o < v < w < i$.
Let us assume that $[v, w]$ is a prime interval. 
Let $u \leq x \prec i$.
Then $y = x \mm w < v$.
\end{lemma}

\begin{proof}
There are three mutually exclusive possibilities: $y \geq v$, 
$ y \parallel v$, and $y < v$.

Since $i = u \jj v$,
we cannot have $y \geq v$, 
because it would imply that $x = u \jj v$.

We want to prove that $y < v$. 
So by way of contradiction, let us assume that 
\begin{equation}\label{E:assume}
   y \parallel v,
\end{equation}
see the first diagram of~Figure~\ref{F:newlemma}.
Define the elements $o_l$ and $o_r$ satisfying
\begin{equation}\label{E:start}
o \prec o_l \leq u \text{ and } o \prec o_r \leq v,
\end{equation}
see the second diagram of Figure~\ref{F:newlemma}.
Since $u \mm v = 0$, it follows that $o_l \mm o_r = 0$.
Note that 
\begin{equation}\label{E:nleq}
   o_l \nleq y.
\end{equation}
Indeed, if $o_l \leq y$, then $o_l \leq y \leq w$, 
and so $o_l \leq u \mm w = o$,
contradicting~\eqref{E:start}. 

We can further assume that 
\begin{equation}\label{E:leq}
   o_r < y.
\end{equation}
Since $o_r = y$ contradicts \eqref{E:assume},  
if \eqref{E:leq} fails, then $o_r \nleq y$.
Thus $o_l \mm o_r = o_l \mm y = o_r \mm y$,
contradicting Lemma~\ref{L:known}(ii), and
thereby verifying~\eqref{E:leq}. 

So we have \eqref{E:assume}--\eqref{E:leq}.
It follows that $o_r < v$; indeed if $o_r = v$, then $v < y$, contradicting \eqref{E:assume}. Let $z = y \mm v$.
Since $o_r < v$, and by \eqref{E:start}, $o_r < y$;
therefore, $o_r \leq z$ and $z < v$ by \eqref{E:assume}.
Also, $z \prec z \jj o_l \neq y$ 
by semimodularity and~\eqref{E:nleq}.
Then $(z \jj o_l) \mm y = (z \jj o_l) \mm v = y \mm v$,
contradicting Lemma~\ref{L:known}(ii). 
\end{proof}

The next statement is a very special case of the Swing Lemma;
it is also a crucial step in its proof.
We are considering the following condition  
for a slim patch lattice~$K$:

\begin{enumeratei}
\item[(SL)]
Let $\fq$ be a prime interval of $K$ 
on the lower right boundary of $K$, 
that is, let $1_\fq \leq c_\tup{r}$.
Then there exists a sequence of prime intervals
$\fp_\tup{l} = \fr_0, \fr_1, \dots, \fr_n = \fq$
such that $\fr_i$ is down-perspective to 
or swings to $\fr_{i+1}$ for $i = 0, \dots, n-1$. 
\end{enumeratei}

\begin{lemma}\label{L:SPSprojspec}
Let $K$ be a slim patch lattice and let
$S =\set{o, \al, \ar, t}$ be a covering square of~$K$,
with $\al$ to the left of $\ar$.
If \tup{(SL)} hold in~$K$, then \tup{(SL)} 
also holds in $K[S]$.
\end{lemma}

\begin{proof}
Note that $\fp_\tup{l}$ and $\fp_\tup{r}$ are also the
two prime intervals of $K[S]$ on the top boundaries of $K[S]$;
the elements $c_\tup{l}$ and $c_\tup{r}$ are also remain the same..

To verify (SL) for $K[S]$, 
let $\fq$ be a prime interval of $K[S]$ 
on the lower right boundary, that is, $1_\fq \leq c_\tup{r}$.

If~$K = \SB 2$, then (SL) is trivial because $K[S] = \SN 7$. 
So we can assume that $K \not\iso \SB 2$.

There are two cases to consider.
 
\emph{Case 1: $\fq \ci K$.} 
Since $\fq$ is prime in $K[S]$ and $\fq \ci K$, 
it follows that $\fq$ is prime in~$K$.
So we can apply (SL) to $\fq$ in $K$, to obtain 
a shortest sequence of prime intervals in $K$
and a~sequence of binary relations 
\begin{equation}\label{E:seque}
\fp_\tup{l} = \fr_0 \grrel_1 \fr_1 \grrel_2 \fr_2 
      \dots \grrel_n \fr_n = \fq,
\end{equation}
where each relation $\grrel_i$ 
is $\perspdn$ or $\swing$.
If all the $\fr_i$ are prime intervals 
in $K[S]$, then the sequence \eqref{E:seque}
verifies (SL) in $K[S]$ for $\fq$. 
So let some $\fr_i$ not be prime in $K[S]$;
we choose the $\fr_j$ with the largest $j$ 
so that $\fr_j$ is not a~prime. 
Since no element of $F[S]$ (defined in \eqref{E:FS},
see also Figure~\ref{F:forkdetails}) can be 
on the upper left boundary of~$L$, we conclude that $j \neq 0$. 
Since $\fq$ is prime in~$K$, it follows that $j \neq n$. 
Therefore, 
\begin{equation}\label{E:0jn}
0 < j < n
\end{equation}
and the intervals $\fr_{j+1}, \dots, \fr_n = \fq$ are prime in $K[S]$,
while the interval $\fr_{j}$ is not. 

There are two possibilities:   
$\fr_{j} \perspdn \fr_{j+1}$ 
or $\fr_{j} \swing \fr_{j+1}$ in~$K$---note that 
$j+1 \leq n$ by \eqref{E:0jn}.
If $\fr_{j} \perspdn \fr_{j+1}$ in~$K$, then
$\fr_{j+1} \perspup \fr_{j}$ in $K$ and in $K[S]$. 
Since $\fr_{j+1}$ is prime in $K[S]$ but 
$\fr_{j}$ is not, this conflicts with
the semimodularity of $K[S]$.
We conclude that $\fr_{j} \swing \fr_{j+1}$.
Let $\fr_{j} \swing \fr_{j+1}$ be established by 
an $\SN 7$ generated by $0_{\fr_{j}}, 0_{\fr_{j+1}}, w$,
where $w$ is the right-most element covered by $1_{\fr_{j}}$
if $\fr_{j}$ is to the left of $\fr_{j+1}$
and the left-most element covered by $1_{\fr_{j}}$, otherwise.
Note that $\set{0_{\fr_{j}}, 0_{\fr_{j+1}}, w}$ is a three-element set
since $\fr_{j} \swing \fr_{j+1}$.
 
Since $\fr_{j}$ is not prime in $K[S]$, it follows that
$0_{\fr_{j}} < z < 1_{\fr_{j}}$ for some $z \in F[S]$. 
We~cannot have $z = m$, because in $K[S]$,
$z$ is contained in an interval $[0_{\fr_{j}}, 1_{\fr_{j}}]$
that is prime in $K$, while $m$ is not
contained in an interval that is prime in $K$.
We~conclude that 
\begin{equation}\label{E:zdef}
   z = z_{r, p}, \text{ for some $1 \leq p \leq n_r$},
\end{equation}
or symmetrically.
It follows that $[o, \ar] \perspdn \fr_j$ in $K$
and so 
\begin{equation}\label{E:perspdn}
   [\br, \ar] \perspdn [z, 1_{\fr_{j}}].
\end{equation}
Since $\fr_{j} \swing \fr_{j+1}$ in $K$, 
it follows that $1_{\fr_j}$ covers 
at least three elements in $K$,
and so~$1_{\fr_j}$ covers 
at least three elements in $K[S]$.
Therefore, in $K[S]$,
\begin{equation}\label{E:swing}
   [z, 1_{\fr_j}] \swing \fr_{j+1}.
\end{equation}
Since $[o, \ar] \perspdn \fr_{j}$ and
$\fr_{j-1} \perspdn \fr_{j}$,
we can apply Lemma~\ref{L:persp} to conclude that either  
\begin{equation}\label{E:2step1}
   [o, \ar] \perspdn \fr_{j-1} \perspdn \fr_{j}   
\end{equation}
or
\begin{equation}\label{E:2step2}
   \fr_{j-1} \perspdn [o, \ar] \perspdn \fr_{j}.   
\end{equation}

If \eqref{E:2step1} holds, then $\fp_\tup{l} \neq \fr_{j-2}$,
since $[o, \ar] \perspdn \fr_{j-1}$ and
$[o, \ar]$ is not on the left boundary of $K[S]$.  
So we have the prime interval $\fr_{j-2}$ satisfying
that $\fr_{j-2} \swing \fr_{j-1}$. 
By Lemma~\ref{L:known}.(v), $\fr_{j-2} \swing \fr_{j-1}$ cannot hold.
So \eqref{E:2step2} holds.

By \eqref{E:0jn}--\eqref{E:swing}, and \eqref{E:2step2},
the sequence of prime intervals with the binary relations
\[
   \fp_\tup{l} = \fr_0 \grrel_1 \fr_1, \dots, \grrel_{j-1} \fr_{j-1}
   \perspdn [\al, t] \swing [m,t] 
   \perspdn [\br \mm 1_{\fr_{j-1}}, 1_{\fr_{j-1}}]
   \swing \fr_{j+1} \grrel_{j+2} \dots \grrel_{n} \fr_n = \fq
\]
establishes (SL) for $K[S]$, see Figure~\ref{F:case1}. 

\begin{figure}[b]
\centerline{\includegraphics{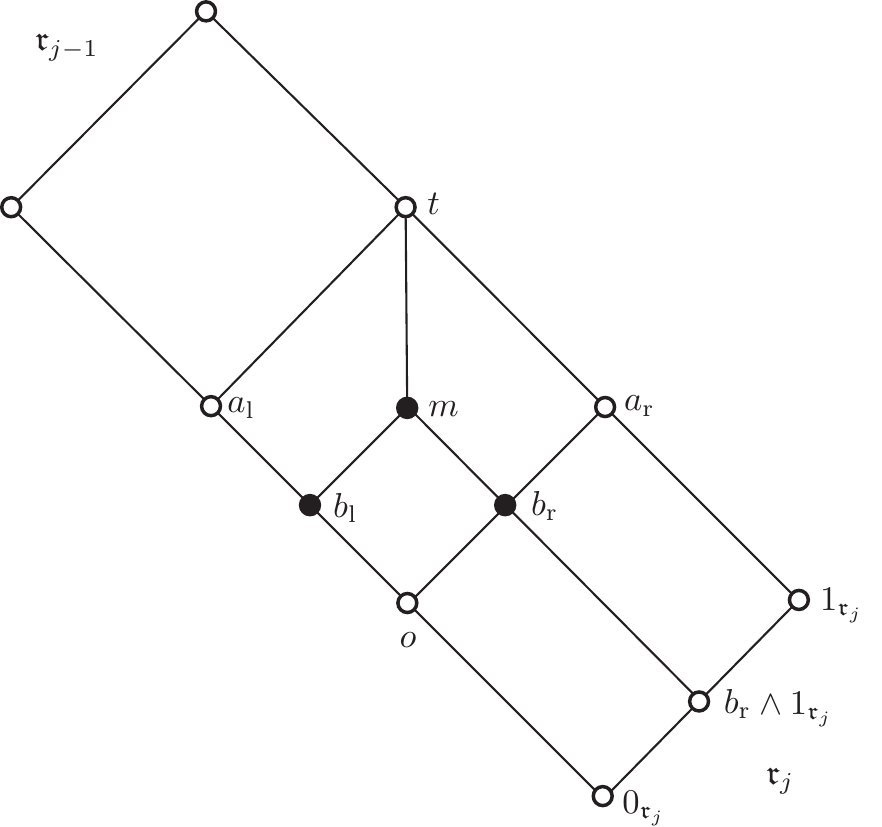}}
\caption{Case 1 of Lemma~\ref{L:SPSprojspec}}\label{F:case1}
\end{figure}

\emph{Case 2: $\fq \nci K$.}
Since $\fq \nci K$ is a prime interval 
on the lower right boundary of~$K[S]$, it follows that
\[
   \fq = [y_{r, n_r}, x_{r, n_r}]_{K[S]} = 
         \set{y_{r, n_r}, z ,x_{r, n_r}},
\] 
where $z = z_{r, n_r}$ 
using the notation of Figure~\ref{F:forkdetails}, or symmetrically.
Let $\qup = [z_{r, n_r}, x_{r, n_r}]$ and 
$\qdn = [y_{r, n_r}, z_{r, n_r}]$;
they are prime intervals in $K[S]$
and $\fq = \qup$ or $\fq = \qdn$.

To verify Case 2, we have to prove (SL) in $K[S]$ for 
$\fq = \qup$ and $\fq = \qdn$.

Let $\fq' = [0_{\qdn}, 1_{\qup}] = [y_{r, n_r}, x_{r, n_r}]$;
it is a prime interval of $K$ on the lower right boundary of $K$.
By applying (SL) to $K$ and $\fq'$, we obtain 
a shortest sequence of prime intervals in $K$
and a~sequence of binary relations 
\begin{equation}\label{E:seque2}
\fp_\tup{l} = \fr_0 \grrel_1 \fr_1 \grrel_2 \fr_2 
      \dots \grrel_n \fr_n = \fq',
\end{equation}
where each relation $\grrel_i$ 
is $\perspdn$ or $\swing$.
Utilizing that the lower right boundary of $K$ is an interval, 
see G. Gr\"atzer and E. Knapp~\cite[Lemma 4]{GKn09}, 
the last step from $\fr_{n-1}$ to $\fr_n = \fq'$
cannot be a swing (if it were, 
$1_\fq$ would cover at least three elements; 
it covers exactly one), 
so $\fr_{n-1} \perspdn \fr_n = \fq'$ holds in $K$.

We have two subcases to consider.

\emph{Case 2a: $n = 1$,} that is, $\fp_\tup{l} = \fr_{n-1}$,
see Figure~\ref{F:case2a}. 
We cannot have $\fp_\tup{l} \swing \fq'$ 
because $\fq'$ is on the lower right boundary of $K$;
therefore, $\fp_\tup{l} \perspdn \fq'$.
We also have $[\al, t] \perspdn \fq'$,
so by Lemma~\ref{L:persp},
we obtain that $\fp_\tup{l} \persp [\al, t]$.
Since $\fp_\tup{l}$ is the top left prime interval of~$K$,
it follows that $\fp_\tup{l} \perspdn [\al, t]$.
Then in $K[S]$, see Figure~\ref{F:forkdetails}, 
\begin{equation}\label{E:xx}
\fp_\tup{l} \perspdn [\al, t] \swing [m, t] \perspdn \qup
\end{equation}
and of course, $\fp_\tup{l} \perspdn \qdn$. 
This completes the verification of (SL) for $\fp_\tup{l}$ and $\fq$.

\emph{Case 2b: $n > 1$,} and so, $\fp_\tup{l} \neq \fr_{n-1}$.
We conclude that $\fr_{n-1} \perspdn \fr_{n} = \fq'$.
Since $[o,\ar] \perspdn \fr_{n} = \fq'$ also holds,
we use Lemma~\ref{L:persp} to obtain that
\begin{equation}\label{E:dnpersp}
\fr_{n-1} \perspdn [o,\ar]
\end{equation}
or 
\begin{equation}\label{E:dnpersp2}
[o,\ar] \perspdn \fr_{n-1}.
\end{equation}
But \eqref{E:dnpersp2} would imply that $0_{\fr_{n-1}}$
is meet-reducible, contradicting that 
$0_{\fr_{n-1}}$ is not the left-most or right-most element
covered by $1_{\fr_{n-2}} = 1_{\fr_{n-1}}$.
We conclude that \eqref{E:dnpersp} holds.

Then $\fr_{n-2} \swing \fr_{n-1}$ 
and $1_{\fr_{n-2}} = 1_{\fr_{n-1}}$ by the definition 
of the swing relation. 
The~element $1_{\fr_{n-2}} = 1_{\fr_{n-1}}$ 
covers at least three elements
and $0_{\fr_{n-1}}$ is not the left-most or right-most element
covered by $1_{\fr_{n-2}} = 1_{\fr_{n-1}}$. 
We can also assume that 
$0_{\fr_{n-21}}$ is to the right of $0_{\fr_{n-2}}$ 
and the down-perceptivity $\fr_{n-1} \perspdn \fr_{n}$ 
is also to the right, as in Figure~\ref{F:case2b}.
Then $\fr_{n-1} \perspdn \qdn$
in $K[S]$.
So the sequence
\[
   \fp_\tup{l} = \fr_0, \fr_1, \dots, \fr_{n-1}, \qdn
\]
verifies (SL) for $\fp_\tup{l}$ and $\qdn$.
\begin{figure}[p]
\centerline{\includegraphics[scale=1]{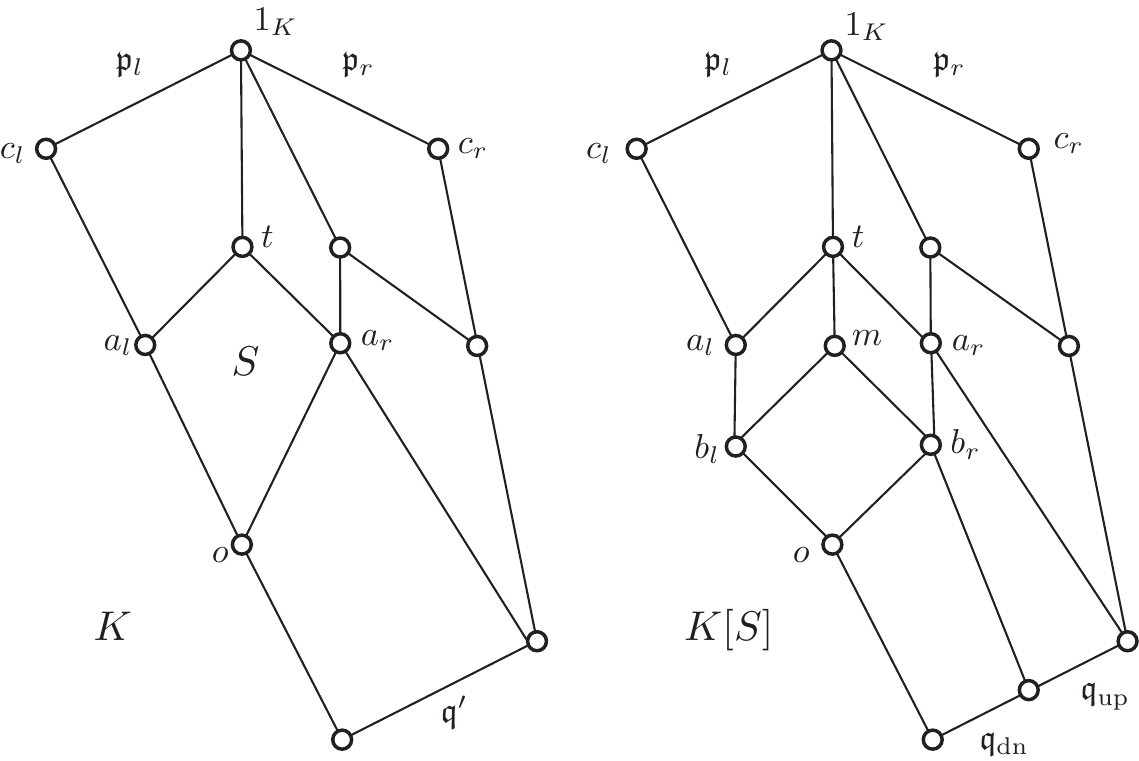}}
\caption{Case 2a of Lemma~\ref{L:SPSprojspec}}
\label{F:case2a}

\bigskip

\bigskip

\centerline{\includegraphics[scale=1]{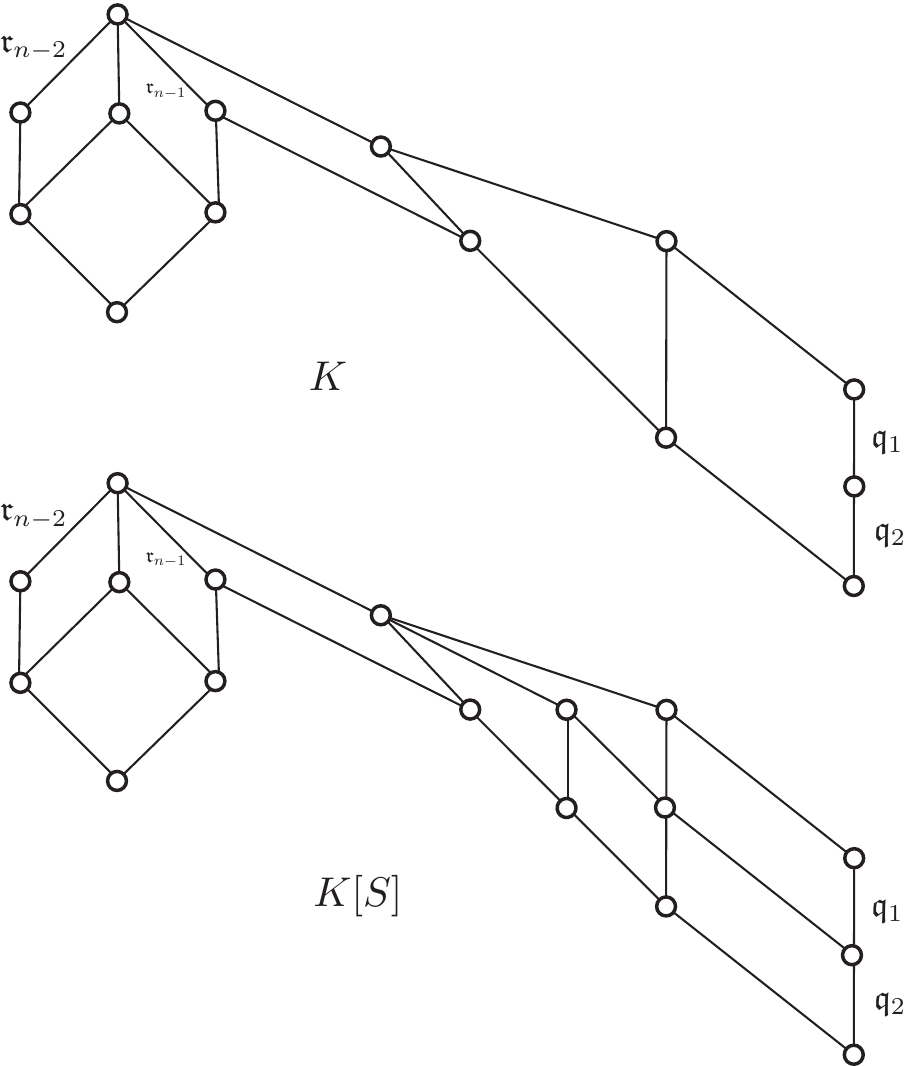}}
\caption{Case 2b of Lemma~\ref{L:SPSprojspec}}
\label{F:case2b}
\end{figure}
\end{proof}

\section{Proving the Swing Lemma}\label{S:Proving}
The following lemma almost yields the Swing Lemma.

\begin{lemma}\label{L:Swingstrong}
Let $L$ be an SPS lattice 
and let $\fp$ and $\fq$ be distinct prime intervals in~$L$.
If $\fp \pperspdn \fq$, then 
there exists a sequence of pairwise distinct prime intervals
\begin{equation}\label{Eq:sequence2}
\fp = \fr_0, \fr_1, \dots, \fr_n = \fq
\end{equation}
such that $\fr_i$ is down-perspective to or swings to $\fr_{i+1}$
for $i = 0, \dots, n-1$. 
\end{lemma}

\begin{proof}
Let $\fp \pperspdn \fq$.
If $\fp \perspdn \fq$ holds, then the statement is trivial.
If $\fp \perspdn \fq$ fails to hold, 
then we induct on the length of the interval $[1_\fq, 1_\fp]$,
in formula, $\length[1_\fq, 1_\fp]$. 

\emph{For the induction base,} let $\length[1_\fq, 1_\fp] = 1$.
Let $L' $ be the interval $[0_\fp \mm 1_\fq, 1_\fp]$ of $L$.
Note that $\fq \ci [0_\fp \mm 1_\fq, 1_\fp]$ and
$\fp \pperspdn \fq$ in $L'$.  
Since $L'$ is a slim patch lattice,
by the Structure Theorem for Slim Patch Lattices, 
we can obtain $L'$ from the planar distributive lattice $D = \SB 2$
by a series of fork insertions. 
Since $D$ has property (SL) and  
fork insertions preserve (SL) by Lemma~\ref{L:SPSprojspec},
it follows that (SL) hold in $L'$. 
So we obtain in $L'$ the sequence \eqref{Eq:sequence2},
which of course, will serve in $L$ as well.

\emph{For the induction step,} 
let \text{$\length[1_\fq, 1_\fp] >1$.}
So we can choose $1_\fq < c \prec 1_\fp$. 
Let $a = 0_\fp \mm c$ and $d = 0_\fp \mm 1_\fq$,
see Figures~\ref{F:proof1} and \ref{F:induct}, 
where the five black filled elements form a sublattice $\SN 5$
establishing that $\fp \ppersp \fq$.
Note that by assumption $d < 0_\fq$.

There are two cases to consider.                                            

\begin{figure}[t!]
\centerline{\includegraphics{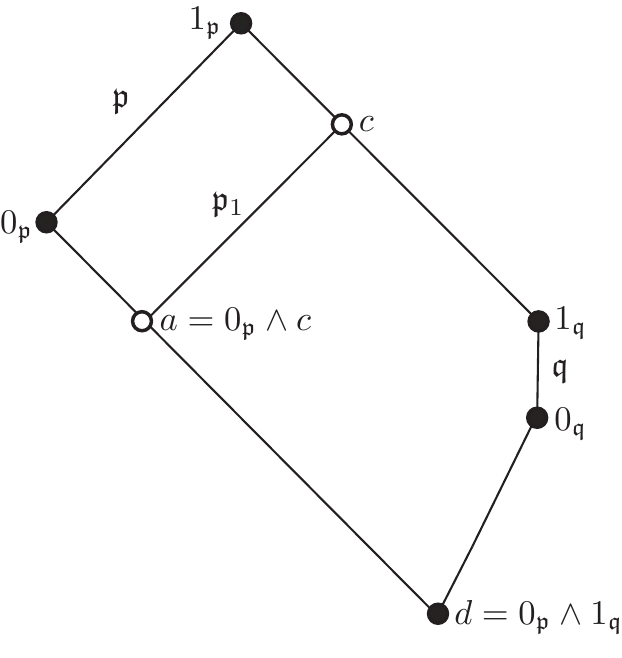}}
\caption{The elements for the inductive step, Case 1
in Lemma~\ref{L:Swingstrong}}
\label{F:proof1}
\end{figure}

\emph{Case 1: $[a, c]$ is a prime interval.}
Let $\fp_1 = [a, c]$, see Figure~\ref{F:proof1}. 
We claim that $\fp_1 \ppersp \fq$.
Indeed, $1_{\fp_1} = c > 1_\fq$ and 
\begin{equation}\label{E:long}
0_{\fp_1} \mm 1_\fq = a \mm 1_\fq 
= (0_\fp \mm a) \mm 1_\fq = a \mm (0_\fp \mm 1_\fq) 
= a \mm d = d < 0_\fq.
\end{equation}
If $a = a \jj 0_\fq$, then $0_\fp \jj 0_\fq = 0_\fp$, 
in conflict with the assumption that $\fp \pperspdn \fq$.
So $a < a \jj 0_\fq \leq c$; 
since $[a, c]$ is assumed to be a prime interval, 
it follows that $a \jj 0_\fq = c$.
Along with \eqref{E:long},
this verifies that $\fp_1 \ppersp \fq$.
Since 
\[
   \length[1_\fq, 1_{\fp_1}] < \length[1_\fq, 1_\fp],
\] 
by the inductive hypothesis, we conclude that 
$\fp_1 \ppersp \fq$.  
Combining this relation with $\fp \perspdn \fp_1$, 
we obtain \eqref{Eq:sequence2}, completing the proof for Case 1.

\emph{Case 2: $[a, c]$ is not a prime interval.}
Let $e = a \jj 1_\fq \leq c$.
Choose an element $b$ so that $a < b \prec c$, 
see Figure~\ref{F:induct}, and let $\fp_1 = [b, e]$.
Then 
\begin{equation}\label{E:stepone}
   \fp \pperspdn \fp_1 
\end{equation}
established by the $\SN 5 = \set{a, 0_\fp, 1_\fp, b, e}$.
\begin{figure}[hbt]
\centerline{\includegraphics{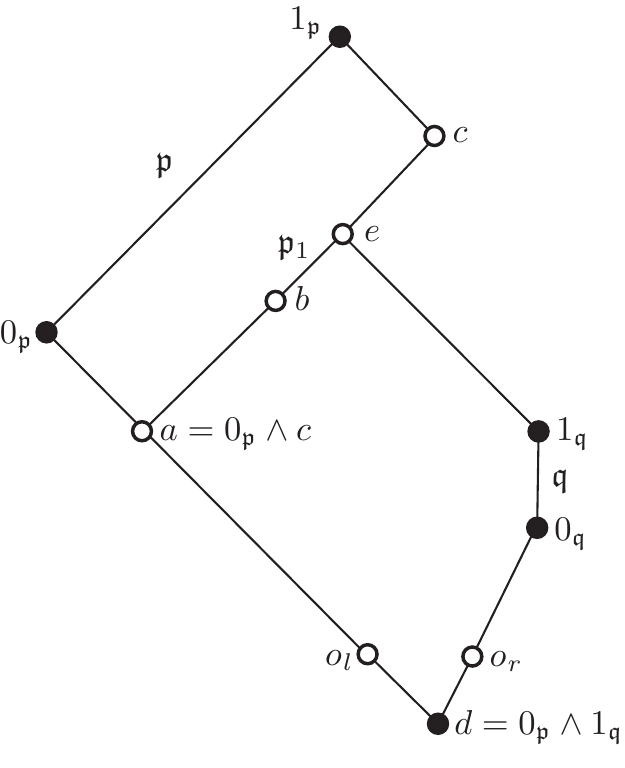}}
\caption{The elements for the inductive step, Case 2}
\label{F:induct}
\end{figure}
We apply Lemma~\ref{L:newlemma} 
with $\fp_1 = [x, i]$, $a = u$, $\fq = [v, w]$, and $o = d$.
Then we conclude that $0_{\fp_1} \mm 1_\fq = b \mm 1_\fq < 0_\fq$,
therefore, 
\begin{equation}\label{E:ppersp}
\fp_1 \pperspdn \fq.
\end{equation}
Now \eqref{E:stepone} and \eqref{E:ppersp} 
imply that $\fp \pperspdn \fq$, which we are required to prove.
\end{proof}

Now we are ready to prove the Swing Lemma.
Let $L$ be an SPS lattice 
and let $\fp$ and $\fq$ be distinct prime intervals in $L$ so that 
$\fq$ is collapsed by $\con{\fp}$. 
By ~the Prime-projectivity Lemma,
there exists a sequence of pairwise distinct prime intervals
$\fp = \fu_0, \fu_1, \dots, \fu_n = \fq$ satisfying
\begin{equation}\label{E:ppthm2}
\fp = \fu_0 \ppersp \fu_1 \ppersp \dotsm \ppersp \fu_n = \fq.
\end{equation}
If $\fu_{i-1} \pperspup \fu_i$ for $i = 1, \dots, n$, 
then $\fu_{i-1} \perspup \fu_i$ by semimodularity. 
If $\fu_{i-1} \pperspdn \fu_i$ for $i = 1, \dots, n$, 
then by Lemma~\ref{L:Swingstrong}, 
we get a sequence of down perspectivities and swings.
So \eqref{E:ppthm2} turns into a sequence of up perspectivities,
down perspectivities, and swings. By Lemma~\ref{L:known}(i)
(or Lemma~\ref{L:persp}), a down perceptivity
cannot be followed by an up perceptivity. 
By Lemma~\ref{L:known}.(v), a swing cannot be followed by an up perceptivity.
So if there is an up perceptivity, it must be the first binary relation.
Since two down perspectivities can be replaced by one
and two swings can be replaced by one, 
we conclude that the sequence of binary relations start 
with at most one up perceptivity, 
followed by an alternating sequence of down perspectivities and swings,
as claimed by the Swing Lemma.

\section{Concluding comments}\label{S:Concluding comments}

My paper \cite{gG14c} presents an alternative proof of the Swing Lemma.
G. Cz\'edli applies in \cite{gC14b} 
the Trajectory Coloring Theorem for Slim Rectangular Lattices 
of G.~Cz\'edli~\cite[Theorem 7.3]{gC13} 
to prove the Swing Lemma for rectangular lattices,
which is then extended to SPS lattices in \cite[Lemma 7]{gG14c}.

In \cite[Section 4]{gG14a}, I present 
a number of interesting applications of the Swing Lemma.
For instance, it is proved that coverings 
in the order, $\jcon$, of join-irreducible congruences 
of an SPS lattice $L$ are represented by proper swings.
(A~swing $\fp \swing \fq$ is \emph{proper}, if  
$0_\fp$ is the left-most or the right-most element
covered by~$1_\fp = 1_\fq$.) Only the first swing in \eqref{Eq:sequence}
may be not proper (and only if $\fp = \fr$). 


In \cite[Section 4]{gG14a}, the Swing Lemma is used to proved that 
in an SPS lattice~$L$, the order $\jcon$ has the property that
every element is covered by at most two elements. 
G. Cz\'edli~\cite{gC13a} proves that the converse does not hold.
The order $P$ of Figure~\ref{Fi:DandP} has this property,
but it cannot be represented as $\tup{J}\,(\Con L)$ for any SPS lattice $L$.
See my paper \cite{gG14d} for a different proof.
\begin{figure}[hbt]
\centerline{\includegraphics{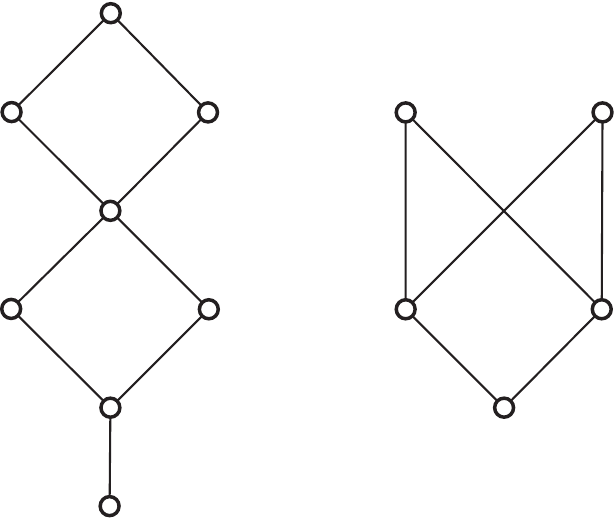}}
\caption{The lattice $D_8$ and the order $P = \tup{J}\,(D_8)$}%
\label{Fi:DandP}
\end{figure}

It would be interesting to see whether the Swing Lemma will be useful
in resolving the problem of characterizing congruence lattices 
of SPS lattices.

\end{document}